\documentclass{ifacconf}

\usepackage{graphicx}      
\graphicspath{{figures/}} 
\usepackage{natbib}        

\usepackage{arydshln}
\usepackage{amsmath}
\usepackage{mathdots}
\usepackage{mathrsfs}  
\usepackage{amssymb} 
\usepackage{arydshln}
\usepackage{graphicx} 
\usepackage{xcolor}
\usepackage{tikz}
\usepackage[absolute]{textpos} 

\newcommand{\R}{\mathbb{R}}         
\newcommand{\N}{\mathbb{N}}         
\renewcommand{\S}{\mathbb{S}}          
\newcommand{\Ac}{\mathcal{A}}

\newcommand{\Gc}{\mathcal{G}}

\newcommand{\Xc}{\mathcal{X}}
\newcommand{\Yc}{\mathcal{Y}}
\newcommand{\Zc}{\mathcal{Z}}

\newcommand{\Ab}{\mathbf{A}}

\newcommand{\Gb}{\mathbf{G}}
\newcommand{\Xb}{\mathbf{X}}
\newcommand{\Zb}{\mathbf{Z}}

\newcommand{\cl}{\prec}
\newcommand{\cg}{\succ}

\newcommand{\diag}{\mathrm{diag}}

\newcommand{\ga}{\gamma}
\newcommand{\al}{\alpha}

\newcommand{\eps}{\varepsilon}

\newcommand{\mat}[2]{\left(\begin{array}{@{}#1@{}}#2\end{array}\right)} 
\newcommand{\smat}[1]{\left(\begin{smallmatrix}#1\end{smallmatrix}\right)}

\newcommand{\teq}[1]{\quad\text{#1}\quad} 
\newenvironment{red_test}{\color{red}}{} 

\newcommand{\h}{\hat}

\newcommand{\mstrut}[1]{\rule{0pt}{#1}}
\let\oldhline=\hline 
\renewcommand{\hline}{\oldhline\mstrut{2.5ex}}
\let\oldhdashline=\hdashline 
\renewcommand{\hdashline}{\oldhdashline\mstrut{2.5ex}}

\renewcommand{\th}{\theta}
\newcommand{\Ag}{{A^c}}
\newcommand{\Bg}{{B^c}}
\newcommand{\Cg}{{C^c}}
\newcommand{\Dg}{{D^c}}
\newcommand{\AJg}{{A_{\!J}^c}}
\newcommand{\BJg}{{B_{\!J}^c}}
\newcommand{\CJg}{{C_{\!J}^c}}
\newcommand{\DJg}{{D_{\!J}^c}}

\newcommand{\Xg}{{\Xc}}
\newcommand{\Xr}{{X}}
\newcommand{\Yr}{{Y}}

\newcommand{\Kr}{{K}}
\newcommand{\Lr}{{L}}
\newcommand{\Mr}{{M}}
\newcommand{\Nr}{{N}}
\newcommand{\KJr}{{K_J}}
\newcommand{\LJr}{{L_J}}
\newcommand{\MJr}{{M_J}}
\newcommand{\NJr}{{N_J}}
\newcommand{\Gr}{{G}}
\newcommand{\Hr}{{H}}
\newcommand{\GJr}{{G_J}}

\newcommand{\GJg}{{\cal G_J}}
\newcommand{\Gg}{{\cal G}}

\newtheorem{theo}{Theorem}
\newtheorem{lemm}[theo]{Lemma}
\newtheorem{coro}[theo]{Corollary}

\newtheorem{rema}[theo]{Remark}

\newenvironment{proof}[1][Proof]{
	\bf #1. \rm}
{\hfill \footnotesize{$\blacksquare$}\vspace{2ex}}

\begin{document}
\begin{frontmatter}

\title{Output-Feedback Synthesis for a Class of Aperiodic Impulsive Systems \thanksref{footnoteinfo}}

\thanks[footnoteinfo]{This project has been funded by Deutsche Forschungsgemeinschaft (DFG, German Research Foundation) under Germany's Excellence Strategy -EXC 2075 -390740016, which is gratefully acknowledged by the authors.}

\author[First]{Tobias Holicki and Carsten W. Scherer}

\address[First]{Department of Mathematics, University of Stuttgart, Germany, 
	\\ e-mail: \{tobias.holicki,~carsten.scherer\}@imng.uni-stuttgart.de}

\begin{abstract}                
We derive novel criteria for designing stabilizing dynamic output-feedback controllers for a class of aperiodic impulsive systems subject to a range dwell-time condition. Our synthesis conditions are formulated as clock-dependent linear matrix inequalities (LMIs) which can be solved numerically, e.g., by using matrix sum-of-squares relaxation methods. We show that our results allow us to design dynamic output-feedback controllers for aperiodic sample-data systems and illustrate the proposed approach by means of a numerical example.
\end{abstract}


\end{frontmatter}



\begin{textblock}{12}(1.75, 15.75)
	\fbox{
		\begin{minipage}{\textwidth}
		\small \textcopyright~ 2022 IFAC. This work has been published in IFAC-PapersOnline under a Creative Commons Licence CC-BY-NC-ND. \\ DOI: 10.1016/j.ifacol.2020.12.981 
		\end{minipage}
}
\end{textblock}


\section{Introduction}

Impulsive systems form a rich class of hybrid system which have applications, e.g., in system biology, robotics as well as communication systems, and which have been studied, e.g., by \cite{GoeSan09}, \cite{HesLib08}, \cite{YeMic98}, \cite{BaiSim89}, \cite{HadChe06} and \cite{Yan01}. They evolve continuously but also undergo instantaneous changes. This leads to a combination of both continuous- and discrete-time dynamics and makes their analysis challenging.
We emphasized that, most interestingly, the class of impulsive systems even encompasses switched and sample-data systems, as shown for example in \cite{Bri17b}, \cite{SivKha94} and \cite{NagHes08}.

In the present paper we consider impulsive systems where the sequence of impulse instants $(t_k)_{k\in \N_0}$ satisfies a range dwell-time condition, i.e., the time distance between two successive jumps is uniformly bounded from below and from above. In particular, the impulses are not restricted to occur in a periodic fashion.
In \cite{HolSch19b} we considered output-feedback gain-scheduling controller synthesis for periodic impulses and added only a few comments on the aperiodic case, which is often more relevant in practice.
The related details are worked out in full detail in this paper. In particular, we provide streamlined and insightful LMI conditions for the design of output-feedback controllers for aperiodic impulsive systems. For reasons of clarity and space, we do not address the extension to gain-scheduling, but emphasize that this is also possible.
Our synthesis procedure relies on a stability result from \cite{Bri13}, which involves so-called clock-dependent LMIs and is well-suited for controller design.
Due to the nature of the analysis result in \cite{Bri13}, the system matrices of the designed impulsive controllers will in general be clock-dependent and thus time-varying. We also propose another analysis result based on a combination of the one from \cite{Bri13} and the so called S-variable approach as extensively discussed in \cite{EbiPea15}; this allows for designing numerically favorable impulsive controllers with constant system matrices.

Output-feedback design results for aperiodic impulsive systems are scarce but can, e.g., be found in \cite{AntHes09}, \cite{MedLaw10}, \cite{Law12} and \cite{ZatPer17}. These rely on separation principles and/or on suitable generalizations of geometric techniques. While \cite{MedLaw10, Law12} focus merely on stabilization, output-feedback regulation is considered in \cite{ZatPer17}. A differential LMI approach to input-output finite-time stabilization is given in \cite{AmaDeT16}. 
Apart from \cite{Law12}, all of the above mentioned papers consider a rather specific structure of the underlying impulsive open-loop system description. In contrast, our design results allow for general linear impulsive systems and can, in particular, be employed for designing controllers for sample-data systems.
Moreover, we go beyond \cite{AmaDeT16} by showing that controller design is possible via parameter elimination, which leads to numerically favorable criteria if compared to a parameter transformation approach, and by providing a systematic procedure for the design of  controllers with constant system matrices.
Finally, we emphasize that our findings on stabilization can be seamlessly extended to more general situations such as gain-scheduling synthesis.

\textbf{Outline.} The present paper is structured as follows. After a short paragraph on notation, we introduce the considered class of impulsive systems and formulate the relevant underlying stability analysis conditions in terms of clock-dependent linear matrix inequalities.
Based on the latter, we derive novel dynamic output-feedback criteria for such impulsive systems by carefully combining several techniques for convexifying synthesis problems.
Afterwards, we demonstrate that our findings even extend to output-feedback design for aperiodic sample-data systems by representing the open-loop interconnection as an impulsive system. Finally, we illustrate our approach with a numerical example. Technical proofs are moved to the appendix.

\textbf{Notation.} $\N$ ($\N_0$) denotes the set of positive (nonnegative) integers and $\S^n$ is the set of symmetric real $n\times n$ matrices.
For a normed space $X$, a function $f:[0, \infty) \to X$ and $t > 0$ we let $f(t^-) := \lim_{s \nearrow t}f(s)$ denote the limit from below once it is well defined; for notational simplicity we set $f(0^-) := f(0)$.
Finally, objects that can be inferred by symmetry or are not relevant are indicated by ``$\bullet$''.

\section{Analysis}
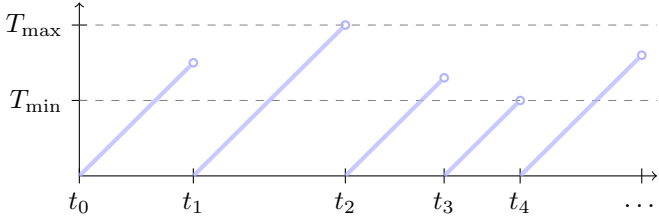
\begin{figure}
	\begin{center}
		\begin{tikzpicture}
		\draw[->] (0, 0) -- (7.6, 0);
		\draw[->] (0, 0) -- (0, 2.3);
		\draw[-] (-0.1, 2) -- node[at start, left]{$T_{\max}$}(0.1, 2);
		\draw[-] (-0.1, 1) -- node[at start, left]{$T_{\min}$}(0.1, 1);
		
		\draw[dashed, very thin, gray] (0, 1) -- (7.6, 1);
		\draw[dashed, very thin, gray] (0, 2) -- (7.6, 2);
		
		\draw[-] (0, -0.1) -- node[at start, below]{$t_0$}(0, 0.1);
		\draw[-] (1.5, -0.1) -- node[at start, below]{$t_1$}(1.5, 0.1);
		\draw[-] (3.5, -0.1) -- node[at start, below]{$t_2$}(3.5, 0.1);
		\draw[-] (4.8, -0.1) -- node[at start, below]{$t_3$}(4.8, 0.1);
		\draw[-] (5.8, -0.1) -- node[at start, below]{$t_4$}(5.8, 0.1);
		\draw[-] (7.4, -0.1) -- node[at start, below, yshift=-1ex]{$\dots$}(7.4, 0.1);
		
		\node[circle, thick, draw=blue!30!white, inner sep=1pt](C1) at (1.5, 1.5){};
		\node[circle, thick, draw=blue!30!white, inner sep=1pt](C2) at (3.5, 2){};
		\node[circle, thick, draw=blue!30!white, inner sep=1pt](C3) at (4.8, 1.3){};
		\node[circle, thick, draw=blue!30!white, inner sep=1pt](C4) at (5.8, 1){};
		\node[circle, thick, draw=blue!30!white, inner sep=1pt](C5) at (7.4, 1.6){};
		\draw[ultra thick, draw=blue!30!white, opacity=0.7] (0, 0) -- (C1);
		\draw[ultra thick, draw=blue!30!white, opacity=0.7] (1.5, 0) -- (C2);
		\draw[ultra thick, draw=blue!30!white, opacity=0.7] (3.5, 0) -- (C3);
		\draw[ultra thick, draw=blue!30!white, opacity=0.7] (4.8, 0) -- (C4);
		\draw[ultra thick, draw=blue!30!white, opacity=0.7] (5.8, 0) -- (C5);
		\end{tikzpicture}
	\end{center}
	\caption{The clock \eqref{ASD::eq::clock} for some $(t_k)_{k\in \N_0}$ satisfying \eqref{ASD::eq::range_dwell_time}.}
	\label{ASD::fig::clock}
\end{figure}

For a sequence of impulse instants $0 = t_0< t_1 < t_2< \dots$ and for some initial condition $x(0) \in \R^n$, let us consider a linear impulsive system with the description
\begin{subequations}
	\label{ASD::eq::sys}
	\begin{alignat}{2}
	\dot x(t) &= A(\th(t)) x(t),\label{ASD::eq::sysa}\\
	x(t_k) &= A_J(\th(t_k^-))x(t_k^-)\label{ASD::eq::sysb}
	\end{alignat}
\end{subequations}
for $t \geq 0$ and $k \in \N$. The function $\th$, which is defined as
\begin{equation}
\th(t) := t - t_k \teq{ for all }t \in [t_k, t_{k+1}) \text{ and }k\in\N_0,
\label{ASD::eq::clock}
\end{equation}
is the so-called clock and depends on the actual sequence of impulse instants $(t_k)_{k\in \N_0}$ as illustrated in Fig.~\ref{ASD::fig::clock}. 
Note that even for systems with constant system matrices, we will design controllers with clock-dependent matrices similarly as in \cite{Bri13}. Since the resulting closed-loop interconnection will be again of the form \eqref{ASD::eq::sys}, we start with presenting analysis conditions for such systems.
\\
In this paper we refer to \eqref{ASD::eq::sys} as impulsive LTV system and as impulse LTI system if the system matrices are constant.

We assume that the sequence $(t_k)_{k \in \N_0}$ satisfies the range dwell-time condition
\begin{equation}
t_{k} - t_{k-1} \in [T_{\min}, T_{\max}] \teq{for all} k \in \N
\label{ASD::eq::range_dwell_time}
\end{equation}
for some fixed $0 < T_{\min} < T_{\max}$. In particular, we do not require the jumps in \eqref{ASD::eq::sys} to appear in a periodic fashion. Other dwell-time conditions such as 
$t_k - t_{k-1} \in [T_{\min}, \infty)$\\ (minimum dwell-time) or $t_k -t_{k-1} = T_{\max}$ (exact dwell-time) for all $k\in \N$ can be handled with minor modifications, but in this paper we focus on \eqref{ASD::eq::range_dwell_time} for clarity.
In the sequel, we assume that $A:[0, T_{\max}] \to \R^{n\times n}$ and $A_J:[T_{\min}, T_{\max}]\to \R^{n\times n}$ are continuous functions, which, together with \eqref{ASD::eq::range_dwell_time}, ensures the existence of a unique piecewise continuously differentiable solution of \eqref{ASD::eq::sys}.

Our clock-dependent design is based on the following stability result that is essentially taken from \cite{Bri13}.

\begin{lemm}
	\label{ASD::lem::stability}
	System \eqref{ASD::eq::sys} is stable, i.e., there exist constants $M, \ga > 0$ such that
	\begin{equation*}
	\|x(t)\|\leq Me^{-\ga t} \|x(0)\|
	\teq{ for all }t\geq 0,
	\end{equation*}
	all initial conditions $x(0) \in \R^n$ and all $(t_k)_{k\in \N_0}$ with \eqref{ASD::eq::range_dwell_time}, if there exists some $X \in C^1([0, T_{\max}], \S^n)$ satisfying
	\begin{subequations}
		\label{ASD::lem::eq::stab}
		\begin{equation}
		X(\tau)\cg 0
		\label{ASD::lem::eq::stab_defi}
		\end{equation}
		and
		\begin{equation}
		\mat{c}{I \\ A(\tau)}^T \mat{cc}{\dot X(\tau) & X(\tau) \\ X(\tau) & 0}\mat{c}{I \\ A(\tau)} \cl 0
		\label{ASD::lem::eq::stab_flow}
		\end{equation}
		for all $\tau \in [0, T_{\max}]$ as well as
		\begin{equation}
		\mat{c}{I \\ A_J(\tau)}^T \mat{cc}{-X(\tau) & 0 \\ 0 & X(0)}\mat{c}{I \\ A_J(\tau)} \cl 0
		\label{ASD::lem::eq::stab_jump}
		\end{equation}
	\end{subequations}
	for all $\tau \in [T_{\min}, T_{\max}]$.
\end{lemm}

Several remarks and additional insights about Lemma \ref{ASD::lem::stability} are given, e.g., in \cite{Bri13}. We merely emphasize that, in contrast to, e.g., lifting or looped-functional based approaches, the conditions \eqref{ASD::lem::eq::stab} are particularly well suited for deriving synthesis criteria as the system matrices $A$ and $A_J$ enter in a convex and very convenient fashion.
Moreover, 
these so-called clock-dependent LMI conditions can be turned into numerically tractable ones by restricting $X$ to be polynomial and by applying the matrix sum-of-squares (SOS) approach (\cite{Par00, SchHol06}). Further note that Lemma~\ref{ASD::lem::stability} can be viewed as a robust analysis result since the conditions \eqref{ASD::lem::eq::stab} guarantee stability for all sequences of impulse instants $(t_k)_{k\in \N_0}$ satisfying \eqref{ASD::eq::range_dwell_time}.

Next to impulsive LTV output-feedback controllers, we also show how to design impulsive LTI controllers. It is well-known in robust control theory the latter requires clock-independent certificates $X(\cdot)$ in Lemma~\ref{ASD::lem::stability}.
Instead of enforcing $X(\cdot)$ to be constant, we rely on the following less conservative analysis result which is based on the S-variable approach, as elaborated on in \cite{EbiPea15} and as originating from \cite{OliBer99}.

\begin{lemm}
	\label{ASD::lem::stability2}
	Suppose that $A$ and $A_J$ are constant. Then \eqref{ASD::eq::sys} is stable for all $(t_k)_{k\in \N_0}$ satisfying \eqref{ASD::eq::range_dwell_time} if there exist $X \in C^1([0, T_{\max}], \S^n)$ and $\rho > 0$, $G, G_J \in \R^{n\times n}$ satisfying
	\begin{subequations}
		\label{ASD::lem::eq::stab2}
		\begin{equation}
		X(\tau)\cg 0
		\label{ASD::lem::eq::stab2_defi}
		\end{equation}
		and
		\begin{equation}
		\arraycolsep=4pt
		\mat{cc}{\dot X(\tau) + A^TG^T + GA & X(\tau) + \rho A^TG^T - G \\
		X(\tau) + \rho GA - G^T & -\rho(G + G^T)} \cl 0
		\label{ASD::lem::eq::stab2_flow}
		\end{equation}
		for all $\tau \in [0, T_{\max}]$ as well as
		\begin{equation}
		\mat{cc}{-X(\tau) & A_J^TG_J^T \\ G_JA_J & X(0) - G_J - G_J^T} \cl 0
		\label{ASD::lem::eq::stab2_jump}
		\end{equation}
	\end{subequations}
	for all $\tau \in [T_{\min}, T_{\max}]$.
\end{lemm}

Following \cite{Oli05}, the proof is based on applying the elimination lemma to eliminate the slack-variables $G$ and $G_J$, which results in the conditions \eqref{ASD::lem::eq::stab}.

Note that the conditions \eqref{ASD::lem::eq::stab2} are more conservative than those in Lemma \ref{ASD::lem::stability}, because the matrix variables $G$, $G_J$ are parameter independent. Equivalence could be retrieved by taking $G$ and $G_J$ to be clock-dependent
(even with a clock-independent $\rho$), but this would prevent the derivation of convex conditions for impulsive LTI controller design. 

\section{Synthesis}

\subsection{Impulsive LTV Controller Design}

For a sequence $(t_k)_{k\in \N_0}$ satisfying \eqref{ASD::eq::range_dwell_time}, some initial condition $x(0)\in \R^n$ and real matrices $A$, $B$, $C$, $A_J$, $B_J$, $C_J$, we now consider an impulsive open-loop system of the form
\begin{subequations}
	\label{ASD::eq::sys_syn}%
	\begin{alignat}{2}
	\mat{c}{\dot x(t) \\ y(t)} &=
	\mat{cc}{A & B \\ C & 0} \mat{c}{x(t) \\ u(t)},
	\label{ASD::eq::sys_syna}\\
	\mat{c}{x(t_k) \\ y_J(k)} &= \mat{cc}{A_J & B_J \\ C_J & 0} \mat{c}{x(t_k^-) \\ u_J(k)}
	\label{ASD::eq::sys_synb}
	\end{alignat}
\end{subequations}
for $t\geq 0$ and $k\in \N$. Here, the signals $u$, $u_J$ and $y$, $y_J$ denote the control inputs and measurement outputs, respectively. Our objective in this subsection is the design of stabilizing dynamic output-feedback controllers for the system \eqref{ASD::eq::sys_syn} and described as
\begin{subequations}
	\label{ASD::eq::con}%
	\begin{alignat}{2}
	\mat{c}{\dot x_c(t) \\ u(t)} &= \mat{cc}{\Ag(\theta(t)) & \Bg(\theta(t))\\ \Cg(\theta(t)) & \Dg(\theta(t))} \mat{c}{x_c(t) \\ y(t)},
	\label{ASD::eq::cona}\\
	\mat{c}{ x_c(t_k) \\ u_J(k)} &= \mat{cc}{\AJg(\th(t_k^-)) & \BJg(\th(t_k^-)) \\ \CJg(\th(t_k^-)) & \DJg(\th(t_k^-))} \mat{c}{x_c(t_k^-) \\ y_J(k)}
	\label{ASD::eq::conb}
	\end{alignat}
\end{subequations}
for $t\geq 0$ and $k\in \N$ with continuous maps $\Ag$, $\Bg$, $\Cg$, $\Dg$, $\AJg$, $\BJg$, $\CJg$, $\DJg$ by relying on Lemma \ref{ASD::lem::stability}. Observe that the interconnection of \eqref{ASD::eq::sys_syn} and \eqref{ASD::eq::con} admits the structure
\begin{subequations}
	\label{ASD::eq::cl_syn}%
	\begin{alignat}{2}
	\dot x_{cl}(t) &= \Ac(\th(t)) x_{cl}(t), \label{ASD::eq::cl_syna}\\
	x_{cl}(t_k) &= \Ac_J(\th(t_k^-))x_{cl}(t_k^-)
	\label{ASD::eq::cl_synb}
	\end{alignat}
\end{subequations}
for $t\geq 0$ as well as $k\in \N$ and with $x_{cl} = \smat{x \\ x_c}$. Here, the maps $\Ac$ and $\Ac_J$ are given by
\begin{equation*}
\arraycolsep=1pt
\mat{cc}{A + B\Dg C & B \Cg  \\
	\Bg C & \Ag }
=
\mat{cc}{A & 0  \\ 0 & 0 } +
\mat{cc}{0 & B \\ I & 0} \mat{cc}{\Ag & \Bg \\ \Cg & \Dg}
\mat{cc}{0 & I\\ C & 0}
\end{equation*}
and
\begin{equation*}
\arraycolsep=0.5pt
\mat{cc}{\hspace{-0.5ex}A_{\!J} \!+\! B_{\!J} \DJg C_J & B_J\CJg\hspace{-0.25ex} \\
	\BJg C_J & \AJg }
\hspace{-0.75ex}=\hspace{-0.75ex}\mat{cc}{\hspace{-0.5ex}A_J & 0\hspace{-0.5ex}  \\ 0 & 0}\hspace{-0.4ex}+\hspace{-0.4ex}\mat{cc}{0 & B_J\hspace{-0.5ex} \\ I & 0}\hspace{-0.5ex}\mat{cc}{\AJg & \BJg \\ \CJg & \DJg}\hspace{-0.5ex}\mat{cc}{0 & I \\ \hspace{-0.5ex}C_J & 0 }\!,
\end{equation*}
respectively. Note that \eqref{ASD::eq::con} can be viewed as a gain-scheduling controller
whose implementation requires the knowledge of the clock-value $\th(t)$ and its left-limit $\th(t^-)$ at time $t$; this is the same as knowing the last jump time $t_k$ with $t_k<t$ and is reminiscent of the approach for static state-feedback controllers in \cite{Bri13}.
Further, observe that we can indeed apply Lemma~\ref{ASD::lem::stability} to \eqref{ASD::eq::cl_syn} since this interconnection is of the form \eqref{ASD::eq::sys}. As usual, trouble arises through the simultaneous search for some certificate $X$ and a controller \eqref{ASD::eq::con} which is a non-convex problem.

A possibility to circumvent this issue is the application of a convexifying parameter transformation that is by now well-known in the LMI literature and has been proposed in \cite{MasOha98} and \cite{Sch96b}. In our case, an extra issue results from the need to apply this transformation on the flow \eqref{ASD::eq::cl_syna} and jump component \eqref{ASD::eq::cl_synb} of the system \eqref{ASD::eq::cl_syn} simultaneously. 

\begin{theo}
	\label{ASD::theo::syn_trafo}
	There exists a controller \eqref{ASD::eq::con} for the system \eqref{ASD::eq::sys_syn} such that the LMIs \eqref{ASD::lem::eq::stab} are feasible for the corresponding closed-loop system if and only if there exist continuously differentiable $\Xr, \Yr$ and continuous $\Kr, \Lr, \Mr, \Nr$, $\KJr, \LJr, \MJr, \NJr$ satisfying
	\begin{subequations}
		\label{ASD::theo::eq::syn_trafo_lmis}%
	    \begin{equation}
		\Xb(\tau) \cg 0
		\label{ASD::theo::eq::syn_trafo_lmisa}
		\end{equation}
		and
		\begin{equation}
		\Zb(\tau) + \Ab(\tau)^T + \Ab(\tau) \cl 0
		\label{ASD::theo::eq::syn_trafo_lmisb}
		\end{equation}
		for all  $\tau \in [0, T_{\max}]$ as well as
		\begin{equation}
		\mat{cc}{\Xb(\tau) & \Ab_J(\tau)^T \\
			\Ab_J(\tau) & \Xb(0)} \cg 0
		\label{ASD::theo::eq::syn_trafo_lmisc}
		\end{equation}
	\end{subequations}
	for all $\tau \in [T_{\min}, T_{\max}]$. Here, the boldface matrix-valued maps are defined as
	\begin{equation*}
	\Xb := \mat{cc}{\Yr & I \\ I & \Xr},\quad
	\Zb := \mat{cc}{- \dot \Yr & 0 \\ 0 &  \dot \Xr},
	\end{equation*}
		\begin{equation*}
	\arraycolsep=2pt
	\Ab
	:= \mat{cc}{A\Yr & A \\ 0 & \Xr A}
	+ \mat{cc}{0 & B \\ I & 0}
	\mat{cc}{\Kr & \Lr \\ \Mr & \Nr}
	\mat{cc}{I & 0  \\ 0 & C}
	\end{equation*}
	and
	\begin{equation*}
	\arraycolsep=1pt
	\Ab_J
	:= \mat{cc}{A_J\Yr& A_J \\ 0 & \Xr(0) A_J}
	+ \mat{cc}{0 & B_J \\ I & 0 }\hspace{-0.5ex}
	\mat{cc}{\KJr & \LJr \\ \MJr & \NJr}\hspace{-0.5ex}
	\mat{cc}{I & 0 \\ 0 & C_J }\hspace{-0.5ex}.
	\end{equation*}
\end{theo}

A constructive proof is given in the appendix. In contrast to the case of periodic impulses considered in \cite{HolSch19b}, the variables $\KJr, \LJr, \MJr, \NJr$ and thus also the system matrices $\AJg$, $\BJg$, $\CJg$, $\DJg$ vary continuously on $[T_{\min}, T_{\max}]$ instead of being constant.
Moreover, observe that the LMIs \eqref{ASD::theo::eq::syn_trafo_lmis} are indeed affine in all decision variables and thus tractable, e.g., by using the SOS approach.

As an alternative, we can utilize the elimination lemma in (\cite{GahApk94, Hel99}) in combination with the continuous selection theorem of \cite{Mic56}. They can either be applied directly to the conditions \eqref{ASD::lem::eq::stab} for the closed-loop system \eqref{ASD::eq::cl_syn} or to the LMIs in Theorem \ref{ASD::theo::syn_trafo}. In particular, we can eliminate almost all of the appearing variables to obtain the following result.

\begin{theo}
	\label{ASD::theo::syn_elim}
	Let $U$, $V$, $U_J$ and $V_J$ be basis matrices of $\ker(B^T)$, $\ker(C)$, $\ker(B_J^T)$ and $\ker(C_J)$, respectively. Then there exists a controller \eqref{ASD::eq::con} for the system \eqref{ASD::eq::sys_syn} such that the LMIs \eqref{ASD::lem::eq::stab} are feasible for the corresponding closed-loop system if and only if there exist continuously differentiable $\Xr, \Yr$ satisfying
	\begin{subequations}
		\label{ASD::theo::eq::syn_elim_lmi}%
		\begin{equation}
		\mat{cc}{\Yr(\tau) & I \\ I & \Xr(\tau)} \cg 0,
		\label{ASD::theo::eq::syn_elim_lmia}%
		\end{equation}
		\begin{equation}
		V^T \mat{c}{I \\ A}^T\mat{cc}{\dot \Xr(\tau) & \Xr(\tau) \\ \Xr(\tau) & 0}\mat{c}{I \\ A} V \cl 0
		\label{ASD::theo::eq::syn_elim_lmib}%
		\end{equation}
		and
		\begin{equation}
		U^T \mat{c}{-A^T \\ I}^T\mat{cc}{0 & \Yr(\tau) \\ \Yr(\tau) & \dot \Yr(\tau)}\mat{c}{-A^T \\ I}  U \cg 0
		\label{ASD::theo::eq::syn_elim_lmic}%
		\end{equation}
		for all $\tau \in [0, T_{\max}]$ as well as
		\begin{equation}
		V_J^T\mat{cc}{I \\ A_J}^T\mat{cc}{-\Xr(\tau) & 0 \\ 0 & \Xr(0)}\mat{c}{I \\ A_J}V_J \cl 0
		\label{ASD::theo::eq::syn_elim_lmid}%
		\end{equation}
		and
		\begin{equation}
		U_J^T\mat{c}{-A_J^T \\ I}^T \mat{cc}{-\Yr(\tau) & 0 \\ 0 & \Yr(0)}\mat{c}{-A_J^T \\ I}U_J \cg 0
		\label{ASD::theo::eq::syn_elim_lmie}%
		\end{equation}
	\end{subequations}
	for all $\tau \in [T_{\min}, T_{\max}]$.	
\end{theo}

The maps in \eqref{ASD::eq::con} can be reconstructed by building the certificate $\Xc$ as in the proof of Theorem \ref{ASD::theo::syn_trafo} and by pointwise using the elimination lemma  as described in \cite{GahApk94} on the LMIs \eqref{ASD::lem::eq::stab_flow} and \eqref{ASD::lem::eq::stab_jump} for \eqref{ASD::eq::cl_syn}. The continuous selection theorem ensures that it is always possible to obtain continuous maps in \eqref{ASD::eq::con} in this fashion.

\begin{rema}
	\begin{itemize}
		\item Due to the much smaller number of decision variables, it is typically preferable to work with Theorem \ref{ASD::theo::syn_elim} instead of Theorem \ref{ASD::theo::syn_trafo}.
		\item Both theorems can be extended in a straightforward fashion to also incorporate quadratic performance criteria on the flow, jump or mixtures of both components of the resulting closed-loop system.
	\end{itemize}
\end{rema}

\subsection{Impulsive LTI Controller Design}

In this subsection and in contrast to the previous one, our goal is the design of stabilizing impulsive output-feedback controllers for \eqref{ASD::eq::sys_syn} with constant system matrices. This amounts to synthesizing stabilizing controllers with matrices $\Ag$, $\Bg$, $\Cg$, $\Dg$, $\AJg$, $\BJg$, $\CJg$, $\DJg$ and a description%
\begin{subequations}
	\label{ASD::eq::con2}%
	\begin{alignat}{2}
	\mat{c}{\dot x_c(t) \\ u(t)} &= \mat{cc}{\Ag & \Bg\\ \Cg & \Dg} \mat{c}{x_c(t) \\ y(t)},
	\label{ASD::eq::con2a}\\
	\mat{c}{ x_c(t_k) \\ u_J(k)} &= \mat{cc}{\AJg & \BJg \\ \CJg & \DJg} \mat{c}{x_c(t_k^-) \\ y_J(k)}
	\label{ASD::eq::con2b}
	\end{alignat}
\end{subequations}
for $t\geq 0$ and $k\in \N$  The corresponding closed-loop interconnection is then of the form 
\begin{subequations}%
	\label{ASD::eq::cl_syn2}%
	\begin{alignat}{2}%
	\dot x_{cl}(t) &= \Ac x_{cl}(t), \\
	x_{cl}(t_k) &= \Ac_Jx_{cl}(t_k^-)
	\label{ASD::eq::cl_syn2b}
	\end{alignat}
\end{subequations}
for $t\geq 0$ as well as $k\in \N$ and with $x_{cl} = \smat{x \\ x_c}$. The matrices $\Ac$ and $\Ac_J$ are structured as in the previous subsection but do not depend on any parameter. In particular, we can apply Lemma \ref{ASD::lem::stability2} for the stability analysis of \eqref{ASD::eq::cl_syn2}, but recall that this comes along with some conservatism. 
Note that for the implementation of the controller, it is still needed to have precise knowledge about the jump instances $t_k$ up to time $t$ available on-line. It might be possible to circumvent this requirement based on approaches as, e.g., the one given in \cite{XiaXia14} involving techniques from robust control, but this is beyond the scope of this paper.

Similar as before we can apply the convexifying parameter transformation from \cite{DeoGer02} on the LMIs \eqref{ASD::lem::eq::stab2} in Lemma \ref{ASD::lem::stability2} in order to obtain an LMI solution for output-feedback impulsive LTI controller design.

\begin{theo}
	\label{ASD::theo::syn_trafo2}
	There exists a controller \eqref{ASD::eq::con2} for the system \eqref{ASD::eq::sys_syn} such that the LMIs \eqref{ASD::lem::eq::stab2} are feasible for the corresponding closed-loop system if there exist some $\rho > 0$, a continuously differentiable $\Xb$ and matrices $\Gr$, $\Hr$, $S$, $\GJr$, $S_J$ as well as  $\Kr, \Lr, \Mr, \Nr$, $\KJr, \LJr, \MJr, \NJr$ such that
	\begin{subequations}
		\label{ASD::theo::eq::syn_trafo_lmis2}%
		\begin{equation}
		\Xb(\tau) \cg 0
		\label{ASD::theo::eq::syn_trafo_lmis2a}
		\end{equation}
		and
		\begin{equation}
		\mat{cc}{\dot \Xb(\tau) + \Ab^T + \Ab & \Xb(\tau) + \rho \Ab^T - \Gb \\ \Xb(\tau)+\rho \Ab - \Gb^T & -\rho(\Gb + \Gb^T)} \cl 0
		\label{ASD::theo::eq::syn_trafo_lmis2b}
		\end{equation}
		for all  $\tau \in [0, T_{\max}]$ as well as
		\begin{equation}
		\mat{cc}{-\Xb(\tau) & \Ab_J^T \\ \Ab_J & \Xb(0) - \Gb_J-\Gb_J^T} \cl 0
		\label{ASD::theo::eq::syn_trafo_lmis2c}
		\end{equation}
	\end{subequations}
	for all $\tau \in [T_{\min}, T_{\max}]$. Here, the boldface matrices $\Gb$, $\Gb_J$, $\Ab$, $\Ab_J$ are defined as
	\begin{equation*}
	\Gb := \mat{cc}{\Hr & I \\ S & \Gr},\quad
	\Gb_J := \mat{cc}{\Hr & I \\ S_J & \GJr},
	\end{equation*}
	\begin{equation*}
	\arraycolsep=2pt
	\Ab
	:= \mat{cc}{A\Hr & A \\ 0 & \Gr A}
	+ \mat{cc}{0 & B \\ I & 0}
	\mat{cc}{\Kr & \Lr \\ \Mr & \Nr}
	\mat{cc}{I & 0  \\ 0 & C}
	\end{equation*}
	and
	\begin{equation*}
	\arraycolsep=1pt
	\Ab_J
	:= \mat{cc}{A_J\Hr & A_J \\ 0 & \GJr A_J}
	+ \mat{cc}{0 & B_J \\ I & 0 }\hspace{-1ex}
	\mat{cc}{\KJr & \LJr \\ \MJr & \NJr}\hspace{-1ex}
	\mat{cc}{I & 0 \\ 0 & C_J }.
	\end{equation*}
\end{theo}

The proof proceeds along the lines of the one of Theorem \ref{ASD::theo::syn_trafo} and is sketched as follows. Once the LMIs \eqref{ASD::theo::eq::syn_trafo_lmis2} are feasible, we can find $U$, $U_J$ and $V$ satisfying
\begin{equation*}
\Gr \Hr +UV^T = S
\teq{ and }
\GJr \Hr + U_JV^T = S_J,
\end{equation*}
such as $V := \Hr^T$, $U := S\Hr^{-1} - \Gr$ and $U_J := S_J\Hr^{-1} - \GJr$. The controller matrices in \eqref{ASD::eq::con2a}, \eqref{ASD::eq::con2b} are then given by
\begin{equation*}
\mat{cc}{U & \Gr B \\ 0 & I}^{-1} \mat{cc}{\Kr - \Gr A\Hr & \Lr \\ \Mr & \Nr} \mat{cc}{V^T & 0 \\ C\Hr & I}^{-1},
\end{equation*}
\begin{equation*}
\mat{cc}{U_J & \GJr B_J \\ 0 & I}^{-1} \mat{cc}{\KJr - \GJr A_J \Hr & \LJr \\ \MJr & \NJr} \mat{cc}{V^T & 0 \\ C_J \Hr & I}^{-1}.
\end{equation*}
Finally, \emph{all} three inequalities \eqref{ASD::theo::eq::syn_trafo_lmis2} are converted by congruence transformations based on $\Yc^{-1}:= \smat{\Hr & I \\ V^T & 0}^{\!-\!1}$ into the three LMIs~\eqref{ASD::lem::eq::stab2} for the closed-loop system \eqref{ASD::eq::cl_syn2} and the variables
$\Gg:=\smat{G & SH^{-\!1}-G \\  H^{-\!T}\!- G & G-SH^{-\!1}}$, 
$\GJg:=\smat{\GJr & S_JH^{-\!1}-G_J \\ H^{-\!T}\!-G_J & G_J - S_JH^{-\!1}}$ as well as $\Xc := \Yc^{-T}\Xb \Yc^{-1}$.

In contrast to \cite{DeoGer02} for multi-objective control, we work with matrices $\Gc$ and $\Gc_J$ that are only partially coupled with an identical choice of the block $H$. In particular, we do \emph{not} require the equality $\Gc = \Gc_J$ which reduces conservatism. 
Note that it is not possible to completely avoid a coupling between $\Gc$ and $\Gc_J$ for synthesis based on parameter transformations, as $\Xc$ appears in both LMIs \eqref{ASD::lem::eq::stab2_flow} and \eqref{ASD::lem::eq::stab2_jump}. 
This is also why the conditions in Theorem~\ref{ASD::theo::syn_trafo2} are no longer necessary, which is in contrast to the clock-dependent design criteria in Theorem \ref{ASD::theo::syn_trafo}.

Note that it is generally not possible to eliminate the constant matrix variables $\Kr, \Lr, \Mr, \Nr$, $\KJr, \LJr, \MJr, \NJr$ from the clock-dependent LMIs \eqref{ASD::theo::eq::syn_trafo_lmis2}
since their reconstruction would require a robust version of the elimination lemma. Unfortunately, such a version is only available in specific situations as pointed out by \cite{Oli05}.


\section{Application to Sample-Data Systems}

For a sequence $(t_k)_{k\in \N_0}$ satisfying \eqref{ASD::eq::range_dwell_time}, real matrices and some initial condition $x(0) \in \R^n$, we now consider a system
\begin{subequations}
	\label{ASD::eq::sd_sys}%
\begin{equation}
\dot x(t) = Ax(t) + Bu(t),\quad y_J(k) = C_J x(t_k^-)
\label{ASD::eq::sd_sysa}
\end{equation}
with the control input $u$ being restricted as
\begin{equation}
u(t) = u(t_k)
\text{ ~for all~ } t\in [t_k, t_{k+1})
\text{ and }k\in \N_0.
\label{ASD::eq::sd_sysb}%
\end{equation}
\end{subequations}
In particular, only output samples are available for control and the control input is the result of a zero-order-hold operation. It is well-known that such sampled-data systems can be reformulated as impulsive systems. This enables us to perform dynamic output-feedback controller design for such systems with aperiodic sampling times based on the presented results with ease.
To this end, the condition \eqref{ASD::eq::sd_sysb} is handled by viewing $u$ as an additional state. This allows us to reformulate the system \eqref{ASD::eq::sd_sys} as
\begin{subequations}
	\label{ASD::eq::sd_sye}%
	\begin{alignat}{2}
	\mat{c}{\dot x(t) \\ \dot{u}(t) \\ \hline y(t)} &=
	\mat{cc|c}{A & B & 0 \\ 0 & 0 & 0 \\ \hline 0 & 0 & 0} \mat{c}{x(t) \\ u(t) \\ \hline \h u(t)},
	\label{ASD::eq::sd_syea}\\
	\mat{c}{x(t_k) \\ u(t_k) \\ \hline y_J(k)} &= \mat{cc|c}{I & 0 & 0 \\ 0 & 0 & I \\ \hline C_J & 0 & 0} \mat{c}{x(t_k^-) \\  u(t_k^-) \\ \hline u_J(k)}
	\label{ASD::eq::sd_syeb}
	\end{alignat}
\end{subequations}
for all $t \geq 0$ and all $k \in \N$, which is clearly a special case of the  description \eqref{ASD::eq::sys_syn}. This immediately leads to the following result which is a consequence of Theorem~\ref{ASD::theo::syn_elim}; Theorems \ref{ASD::theo::syn_trafo} and \ref{ASD::theo::syn_trafo2} could be employed here in exactly the same fashion.

\begin{coro}
	\label{ASD::coro::syn_elim}
	Let  $\h A \!=\! \smat{A & B \\ 0 & 0}$ and let $V_J$ be a basis matrix of $\ker(C_J)$. Then there exists a controller \eqref{ASD::eq::con} for \eqref{ASD::eq::sd_sye} such that the LMIs \eqref{ASD::lem::eq::stab} are feasible for the corresponding closed-loop system if and only if there exist continuously differentiable $\Xr = \smat{X_1 & \bullet \\ \bullet & \bullet}, \Yr = \smat{Y_1 & \bullet \\ \bullet & \bullet}$ satisfying
	\begin{subequations}
		\label{ASD::coro::eq::syn_elim_lmi}%
		\begin{equation}
		\mat{cc}{\Yr(\tau) & I \\ I & \Xr(\tau)} \cg 0,
		\label{ASD::coro::eq::syn_elim_lmia}%
		\end{equation}
		\begin{equation}
		\mat{c}{I \\ \h A}^T\mat{cc}{\dot \Xr(\tau) & \Xr(\tau) \\ \Xr(\tau) & 0}\mat{c}{I \\ \h A} \cl 0
		\label{ASD::coro::eq::syn_elim_lmib}%
		\end{equation}
		and
		\begin{equation}
		\mat{c}{ -\h A^T \\ I}^T\mat{cc}{0 & \Yr(\tau) \\ \Yr(\tau) & \dot \Yr(\tau)}\mat{c}{-\h A^T \\ I} \cg 0
		\label{ASD::coro::eq::syn_elim_lmic}%
		\end{equation}
		for all $\tau \in [0, T_{\max}]$ as well as
		\begin{equation}
		\mat{cc}{V_J & 0 \\ 0 & I}^T\left( \mat{cc}{X_1(0) & 0 \\ 0 & 0} - X(\tau)\right)\mat{cc}{V_J & 0 \\ 0 & I} \cl 0
		\label{ASD::coro::eq::syn_elim_lmid}%
		\end{equation}
		and
		\begin{equation}
		Y_1(0) - Y_1(\tau) \cg 0
		\label{ASD::coro::eq::syn_elim_lmie}%
		\end{equation}
	\end{subequations}
	for all $\tau \in [T_{\min}, T_{\max}]$.	
\end{coro}

Due to the specific structure of \eqref{ASD::eq::sd_sye}, the resulting controller \eqref{ASD::eq::con} can also be expressed as a discrete-time linear time-varying controller of order $n + p$ if $B \in \R^{n \times p}$.

Existing output-feedback design approaches for sampled-data systems are typically based on lifting techniques or on their interpretation as a delay system as, e.g., in \cite{RamMoh14}.
To the best of our knowledge, it is nowhere addressed in the literature apart from \cite{GerCol19} how the representation \eqref{ASD::eq::sd_sys} as an impulsive system can be employed for systematic output-feedback design.
In contrast to \cite{GerCol19}, our underlying design results for impulsive systems are not specially tailored for an application to sample-data systems which makes them more flexible but no more conservative. This flexibility also manifests itself in Theorems \ref{ASD::theo::syn_trafo}, \ref{ASD::theo::syn_elim} and \ref{ASD::theo::syn_trafo2} offering three different design strategies.
Moreover, our conditions easily permit a seamless extension, e.g., to $H_\infty$-performance, to gain-scheduling controller synthesis or to the design of consensus protocols, in parallel to what has been suggested in \cite{HolSch19b}.

\section{Example}

Among the many possibilities to illustrate our results also on concrete applications, we choose one that nicely allows to compare impulsive LTV with impulsive LTI controller designs and to analyze the effect of the hold operation \eqref{ASD::eq::sd_sysb} in terms of conservatism.
To this end, let us consider the family of systems \eqref{ASD::eq::sd_sys} with $T_{\min} = 0.25$,
\begin{equation*}
A = \mat{cc}{0.5 & \al \\ -\al & 0.5},\quad
B = \mat{c}{0 \\ 1}
\teq{ and }
C_J = \mat{cc}{1 & 0}
\end{equation*}
for some parameter $\al$ in $[1, 5]$. With bisection we compute the largest $T_{\max}$, as a function of $\al$, for which we can find a stabilizing controller based on our results. For numerical tractability we search for polynomial matrix functions of degree $4$ and apply an SOS approach with multipliers of degree $2$; a perturbation of the right-hand sides of all inequalities by $-\eps I$ or $\eps I$ with $\eps = 0.1$ ensures strictness of the LMIs. The arising semidefinite programs are solved with \cite{Mos17} and YALMIP (\cite{Lof04}).

The curves resulting from an impulsive LTV (LTI) design for \eqref{ASD::eq::sd_sys} with and without \eqref{ASD::eq::sd_sysb}, respectively, are depicted in Fig.~\ref{ASD::fig::Tmax}.
This illustrates that \eqref{ASD::eq::sd_sysb}, as expected, can be restrictive for larger values of $T_{\max}$ and that there is indeed a cost for designing impulsive LTI controllers \eqref{ASD::eq::con2} instead of ones with clock-dependent system matrices \eqref{ASD::eq::con}.

\begin{figure}
	\begin{center}
		\includegraphics[width=0.49\textwidth, trim=25 208 30 20, clip]{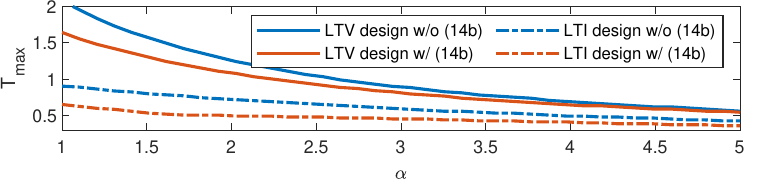}
	\end{center}
	\caption{Largest $T_{\max}$, as a function of $\al$, for which we find an impulsive LTV (LTI) stabilizing controller for \eqref{ASD::eq::sd_sys} with and without \eqref{ASD::eq::sd_sysb}, respectively.}
	\label{ASD::fig::Tmax}
\end{figure}

\section{Conclusion}

We propose a novel streamlined approach for designing stabilizing dynamic output-feedback controller for a class of aperiodic impulsive systems subject to a range dwell-time condition. Our synthesis criteria are based on an analysis result by \cite{Bri13} and formulated as clock-dependent LMIs. These can be solved numerically, e.g., by using matrix SOS relaxations.
We also demonstrate the design of controllers with clock-dependent as well as constant system matrices, and show how our findings can be employed for output-feedback synthesis for aperiodic sample-data systems.
Our findings are illustrated and compared with each other by means of a numerical example.

Future research could, for example, involve studies on output-feedback design in the case that the controller and the underlying system jump in an asynchronous fashion. 

\begin{ack}
We would like to thank Corentin Briat for pointing out several interesting references and a stimulating discussion.
\end{ack}

\bibliography{literatur}  

\appendix
\section{Technical Proofs}    

\begin{proof}[Proof of Theorem \ref{ASD::theo::syn_trafo}]
	We only prove sufficiency as necessity is essentially obtained by reversing the arguments. 
	Whenever we take an inverse of a matrix valued map in the sequel, this is meant pointwise, i.e., for a map $F$ the function $F^{-1}$ satisfies $F^{-1}(\tau)F(\tau) = I$ for all $\tau$ in its domain.
	
	{\bf Step 1: Construction of a Certificate $\Xg$:} Due to \eqref{ASD::theo::eq::syn_trafo_lmisa}, we can infer the existence of differentiable and pointwise nonsingular functions $U, V$ satisfying $UV^T = I - \Xr\Yr$; a possible choice is $U = \Xr$ and $V = \Xr^{-1} - \Yr$. We can then define $\Yc := \smat{\Yr & I \\ V^T & 0}$, $\Zc := \smat{I & 0 \\ \Xr & U}$ and $\Xg := \Yc^{-T} \Zc$.
	
	{\bf Step 2: Transformation of Parameters:} Let us define the controller matrices $\smat{\Ag & \Bg \\ \Cg & \Dg}$ and $\smat{\AJg & \BJg \\ \CJg & \DJg}$ as
	\begin{equation*}
	\mat{cc}{U & \Xr B \\ 0 & I}^{-1} \mat{cc}{\Kr - \Xr A\Yr - \dot \Xr \Yr - \dot U V^T & \Lr \\ \Mr & \Nr} \mat{cc}{V^T & 0 \\ C\Yr & I}^{-1}
	\end{equation*}
	and
	\begin{equation*}
	\arraycolsep=3pt
	\mat{cc}{U(0) & \Xr(0)B_J \\ 0 & I}^{\!-1}\! \mat{cc}{\KJr - \Xr(0)A_J \Yr & \LJr \\ \MJr & \NJr} \mat{cc}{V^T & 0 \\ C_J \Yr & I}^{\hspace{-0.5ex}-1},
	\end{equation*}
	respectively. These choices are motivated by the following observations. Note at first that
	\begin{equation*}
	\Yc^T \Xg \Yc = \Yc^T\Yc^{-T}\Zc \Yc = \Zc \Yc = \mat{cc}{\Yr & I \\ I & \Xr} = \Xb
	\end{equation*}
	and
	\begin{equation*}
	\arraycolsep=3pt
	\Yc^T \dot \Xg \Yc
	= \dot \Zc \Yc - \dot \Yc^T \Zc^T
	%
	%
	= \Zb + \mat{cc}{0 & (\bullet)^T \\ \dot \Xr \Yr + \dot U V^T & 0}
	\end{equation*}
	hold since $\Yc^T \Xg = \Zc$ and $\Yc^T \dot \Xg  + \dot \Yc^T \Xg = \dot \Zc$. Moreover, we infer by routine computations that $\Yc^T \Xg \Ac \Yc$ equals
	\begin{equation*}
	\arraycolsep=2pt
	\begin{aligned}
	&\phantom{=} \Zc
	\left[
	\mat{cc}{A & 0  \\ 0 & 0} +
	\mat{cc}{0 & B \\ I & 0}
	\mat{cc}{\Ag & \Bg \\ \Cg & \Dg}
	\mat{cc}{0 & I \\ C & 0}\right]
	\Yc  \\
	%
	%
	&=\hspace{-0.5ex}\mat{cc}{A \Yr & A \\ 0 & \Xr A}\hspace{-0.5ex}\!+\!\hspace{-0.5ex}
	\mat{cc}{0 & B \\ I & 0} \hspace{-1ex}
	\left[\hspace{-0.5ex}\mat{cc}{\Kr & \Lr \\ \Mr & \Nr}\hspace{-0.5ex}-\hspace{-0.5ex}\mat{cc}{\dot \Xr \Yr\!+\!\dot U V^T & 0 \\ 0 & 0}\hspace{-0.5ex} \right]\hspace{-1ex}
	\mat{cc}{I & 0  \\ 0 & C} \\
	&= \Ab - \mat{cc}{0 & 0 \\ \dot \Xr \Yr + \dot U V^T & 0}.
	\end{aligned}
	\end{equation*}
	By combining the last two identities we obtain
	\begin{equation*}
	\Yc^T \big(\dot \Xg + \Ac^T\Xg + \Xg \Ac\big)\Yc
	= \Zb + \Ab^T + \Ab.
	\end{equation*}
	Finally, we compute in a similar fashion
	\begin{equation*}
	\Yc(0)^T \Xg(0)\Ac_J(\tau)\Yc(\tau)
	= \Ab_J(\tau)
	\end{equation*}
	for all $\tau \in [T_{\min}, T_{\max}]$.
	
	{\bf Step 3: Transformation of LMIs:} Due to the identities from the previous step, the LMIs \eqref{ASD::theo::eq::syn_trafo_lmisa} and \eqref{ASD::theo::eq::syn_trafo_lmisb} read, after a congruence transformation with $\Yc^{-1}$, as
	\begin{equation*}
	\Xg \cg 0
	\text{ ~and~ }\dot \Xg + \Ac^T \Xg + \Xg \Ac \cl 0
	\teq{ on } [0, T_{\max}].
	\end{equation*}
	Similarly, a congruence transformation with the matrix $\diag(\Yc(\tau), \Yc(0))^{-1}$ leads from \eqref{ASD::theo::eq::syn_trafo_lmisc} to
	\begin{equation*}
	\mat{cc}{\Xg(\tau) & \Ac_J(\tau)^T\Xg(0) \\
		\Xg(0)\Ac_J(\tau) & \Xg(0)} \cg 0
	\end{equation*}
	and, by an application of the Schur complement, to
	\begin{equation*}
	\Xg(\tau) - \Ac_J(\tau)^T\Xg(0)\Ac_J(\tau) \cg 0
	\text{ ~for all~ }\tau \in [T_{\min}, T_{\max}].
	\end{equation*}
	This finishes the proof.
\end{proof}

\end{document}